\documentclass[final]{dmtcs-episciences}
\usepackage[utf8]{inputenc}
\usepackage{subfigure}
\usepackage[round]{natbib}
\usepackage{amsmath}
\DeclareMathOperator{\MND}{MND}
\DeclareMathOperator{\MNA}{MNA}

\DeclareMathOperator{\asc}{asc}
\DeclareMathOperator{\des}{des}
\DeclareMathOperator{\rlmax}{rlmax}

\DeclareMathOperator{\rlmin}{rlmin}
\DeclareMathOperator{\lrmax}{lrmax}
\DeclareMathOperator{\lrmin}{lrmin}
\usepackage{amsthm}
\newtheorem{thm}{Theorem}[section]
\newtheorem{lem}[thm]{Lemma}

\newtheorem{cor}[thm]{Corollary}

\newtheorem{rem}[thm]{Remark}

\author[Tian Han and Sergey Kitaev]{Tian Han\affiliationmark{1}\thanks{%Email: \texttt{hantian.hhpt@qq.com}. 
		This author acknowledges the support of the 2024 Graduate Research Innovation Project (2024KYCX039Z) of Tianjin Normal University.}
	\and Sergey Kitaev\affiliationmark{2}%\thanks{%Email: \texttt{sergey.kitaev@strath.ac.uk}.}
}
\title[Joint distributions of statistics over permutations]{Joint distributions of statistics over permutations avoiding two patterns of length~3}
\affiliation{
	College of Mathematical Science, Tianjin Normal University, Tianjin, P. R. China\\
	Department of Mathematics and Statistics, University of Strathclyde, Glasgow, United Kingdom}
\begin{document}
\publicationdata{vol. 26:1, Permutation Patterns 2023}{2024}{4}{10.46298/dmtcs.12517}{2023-11-07; 2023-11-07; 2024-06-28; 2024-07-09}{2024-07-27}
	\maketitle
\begin{abstract}
	Finding distributions of permutation statistics over pattern-avoiding classes of permutations attracted much attention in the literature. In particular, Bukata et al. %~\cite{Bukata2019} 
	found distributions of ascents and descents on permutations avoiding any two patterns of length 3. In this paper, we generalize these results in two different ways: we find explicit formulas for the joint distribution of six statistics ($\asc$, $\des$, $\lrmax$, $\lrmin$, $\rlmax$, $\rlmin$), and also explicit formulas for the joint distribution of four statistics ($\asc$, $\des$, $\MNA$, $\MND$) on these permutations in all cases. The latter result also extends the recent studies by Kitaev and Zhang %~\cite{KitZha}
	 of the statistics $\MNA$ and $\MND$ (related to non-overlapping occurrences of ascents and descents) on stack-sortable permutations. All  multivariate generating functions in our paper are rational, and  we provide combinatorial proofs of five equidistribution results that can be derived from the generating functions. 
\end{abstract}

\section{Introduction}\label{sec1}
A permutation of length $n$ is a rearrangement of the set $[n]:=\{1, 2, \ldots, n\}$.  Denote by  $S_n$  the set of permutations of $[n]$. For $\pi \in S_n$, let $\pi^r=\pi_n\pi_{n-1}\cdots\pi_1$ and $\pi^c= (n + 1 - \pi_1)(n + 1 - \pi_2)\cdots (n + 1-\pi_n)$
denote the \emph{reverse} and \emph{complement} of $\pi$, respectively. Then $\pi^{rc}=(n + 1-\pi_n)(n + 1-\pi_{n-1})\cdots (n + 1-\pi_1)$. A permutation $\pi_1\pi_2\cdots\pi_n\in S_n$ avoids a \emph{pattern} $p=p_1p_2\cdots p_k\in S_k$  if there is no subsequence $\pi_{i_1}\pi_{i_2}\cdots\pi_{i_k}$ such that $\pi_{i_j}<\pi_{i_m}$ if and only if $p_j<p_m$. For example, the permutation $32154$ avoids the pattern $231$. 
 Let $S_{n}(\tau, \rho)$ denote the set of permutations in $S_n$ that avoid patterns $\tau$ and $\rho$. The area of permutation patterns attracted much attention in the literature (see \cite{Kitaev2011} and reference therein).

Of interest to us are the following classical permutation statistics. For $1\leq i\leq n-1$, $i$ is an \emph{ascent} (resp., \emph{descent}) in $\pi\in S_n$ if $\pi_i<\pi_{i+1}$ (resp., $\pi_i>\pi_{i+1}$) and $\asc(\pi)$ (resp., $\des(\pi)$) is the number of ascents (resp., descents) in $\pi$.  Also, $\pi_i$ is a \emph{right-to-left maximum} (resp., \emph{right-to-left minimum}) in $\pi$ if  $\pi_i$ is greater (resp., smaller) than any element to its right. Note that $\pi_n$ is always a right-to-left maximum and a right-to-left minimum. Denote by $\rlmax(\pi)$ and $\rlmin(\pi)$ the number of right-to-left maxima and right-to-left minima in $\pi$, respectively. We  define \emph{left-to-right maximum}, \emph{left-to-right minimum}, $\lrmax(\pi)$ and $\lrmin(\pi)$ in a similar way. For example, if $\pi=34152$ then $\lrmax(\pi)=3$ and $\lrmin(\pi)=\rlmin(\pi)=\rlmax(\pi)=\asc(\pi)=\des(\pi)=2$.

We are also interested in the statistics \emph{ maximum number of non-overlapping ascents} (denoted $\MNA$) and  \emph{ maximum number of non-overlapping descents} (denoted $\MND$). For example, $\des(13254)=2=\MND (13254)$ while $3=\des(32154)\ne\MND(32154)=2$. These statistics are a particular case of the study of the maximum number of non-overlapping consecutive patterns in \cite{Kitaev2005} and recently,  %Kitaev and Zhang~
\cite{KitZha} studied $\MNA$ and $\MND$ on permutations avoiding a single pattern of length 3.

Also, $k$-tuples of (permutation) statistics $(s_1, s_2, \ldots , s_k)$ and $(s_1^\prime, s_2^\prime, \ldots , s_k^\prime)$ are
\emph{ equidistributed} over a set $S$ if
$$\sum_{a\in S}t_1^{s_1(a)}t_2^{s_2(a)}\cdots t_k^{s_k(a)}=
\sum_{a\in S}t_1^{s_1^\prime(a)}t_2^{s_2^\prime(a)}\cdots t_k^{s_k^\prime(a)}.$$

There is a line of research in the literature on finding distributions of permutation statistics over pattern-avoiding classes of permutations (see, for example, \cite{Barnabei2010,Barnabei2010-1,Bukata2019,Eliz2004,Eliz2004-2} 
 and references therein). In particular, %Bukata et al.~
\cite{Bukata2019} found distributions of ascents and descents on permutations avoiding any two patterns of length 3. In this paper, we generalize these results in two different ways. Namely, we find explicit formulas for the \emph{ joint} distribution of six statistics ($\asc$, $\des$, $\lrmax$, $\lrmin$, $\rlmax$, $\rlmin$), and also explicit formulas for the \emph{ joint} distribution of four statistics ($\asc$, $\des$, $\MNA$, $\MND$) on these permutations. The latter result also extends recent studies by %Kitaev and Zhang~
\cite{KitZha} of the statistics $\MNA$ and $\MND$ on \emph{ stack-sortable permutations} (which are precisely 231-avoiding permutations). Moreover, we provide \emph{ combinatorial} proofs of five equidistribution results observed from the multi-variable generating functions derived in this paper.

  In what follows, we let g.f.\ stand for ``generating function''.
We  will derive closed form expressions for the following g.f.'s:
$$F_{(\tau, \rho)}(x, p, q, u, v, s, t):=\sum_{n\geq 0}\ \sum_{\pi\in S_{n}(\tau, \rho)}x^{n} p^{\asc(\pi)}q^{\des(\pi)}u^{\lrmax(\pi)}v^{\rlmax(\pi)}s^{\lrmin(\pi)}t^{\rlmin(\pi)},$$
$$G_{(\tau, \rho)}(x, p, q, y, z) := \sum_{n\geq 0}\  \sum_{\pi \in S_{n}(\tau, \rho)}x^np^{\asc(\pi)}q^{\des(\pi)}y^{\MNA(\pi)}z^{\MND(\pi)}$$ 
for all $\tau$ and $\rho$ in $S_3$. All of our g.f.'s are rational functions. Note that
 $$\des(\pi)=\asc(\pi^r)=\asc(\pi^c)=\des(\pi^{rc}),$$
 $$\lrmax(\pi)=\rlmax(\pi^r)=\lrmin(\pi^c)=\rlmin(\pi^{rc}),$$
 $$\MND(\pi)=\MNA(\pi^r)=\MNA(\pi^c)=\MND(\pi^{rc})$$
and hence 
 \begin{align*}
 F_{(\tau^r, \rho^r)}(x, p, q, u, v, s, t)
 &=\sum_{n\geq 0}\  \sum_{\pi\in S_{n}(\tau^r, \rho^r)}x^{n} p^{\asc(\pi)}q^{\des(\pi)}u^{\lrmax(\pi)}v^{\rlmax(\pi)}s^{\lrmin(\pi)}t^{\rlmin(\pi)}\\
 &=\sum_{n\geq 0}\ \sum_{\pi^r\in S_{n}(\tau, \rho)}x^{n} p^{\des(\pi^r)}q^{\asc(\pi^r)}u^{\rlmax(\pi^r)}v^{\lrmax(\pi^r)}s^{\rlmin(\pi^r)}t^{\lrmin(\pi^r)}\\&=F_{(\tau, \rho)}(x, q, p, v, u, t, s);
 \end{align*}
 \begin{align*}
 F_{(\tau^c, \rho^c)}(x, p, q, u, v, s, t)
 &=\sum_{n\geq 0}\ \sum_{\pi\in S_{n}(\tau, \rho)}x^{n} p^{\des(\pi)}q^{\asc(\pi)}u^{\lrmin(\pi)}v^{\rlmin(\pi)}s^{\lrmax(\pi)}t^{\rlmax(\pi)}\\&=F_{(\tau, \rho)}(x, q, p, s, t, u, v);
 \end{align*}
 \begin{align*}
 F_{(\tau^{rc}, \rho^{rc})}(x, p, q, u, v, s, t)
 &=\sum_{n\geq 0}\ \sum_{\pi\in S_{n}(\tau, \rho)}x^{n} p^{\asc(\pi)}q^{\des(\pi)}u^{\rlmin(\pi)}v^{\lrmin(\pi)}s^{\rlmax(\pi)}t^{\lrmax(\pi)}\\&=F_{(\tau, \rho)}(x,  p, q, t,s ,v, u);
 \end{align*}
 \begin{align*}
 G_{(\tau^r, \rho^r)}(x, p, q, y, z) &=G_{(\tau^c, \rho^c)}(x, p, q, y, z) =\sum_{n\geq 0}\ \sum_{\pi \in S_{n}(\tau, \rho)}x^np^{\des(\pi)}q^{\asc(\pi)}y^{\MND(\pi)}z^{\MNA(\pi)}\\&=G_{(\tau, \rho)}(x, q, p, z, y); 
 \end{align*}
 \begin{align*}
 G_{(\tau^{rc}, \rho^{rc})}(x, p, q, y, z) &= \sum_{n\geq 0}\ \sum_{\pi \in S_{n}(\tau, \rho)}x^np^{\asc(\pi)}q^{\des(\pi)}y^{\MNA(\pi)}z^{\MND(\pi)}=G_{(\tau, \rho)}(x, p, q, y, z). 
 \end{align*}
 
 The following results appear in \cite{SIMION1985}.
  
  \begin{thm}\label{lemma-Sim-class}
  Let $A_{n}(\tau, \rho)$ be the number of elements in $S_{n}(\tau, \rho)$. Then, 
  	\begin{itemize}
  		\item[{\upshape (a)}] $A_{n}(123, 132) = A_{n}(123, 213) = A_{n}(321, 231) = A_{n}(321, 312)=2^{n-1}$;
  		\item[{\upshape (b)}] $A_{n}(231, 312)=A_{n}(132, 213)=2^{n-1}$;
  		\item[{\upshape (c)}] $A_{n}(213, 312)=A_{n}(132, 231)=2^{n-1}$;
  		\item[{\upshape (d)}] $A_{n}(213, 231)=A_{n}(132, 312)=2^{n-1}$;
  		\item[{\upshape (e)}] $A_{n}(132, 321) = A_{n}(123, 231) = A_{n}(123, 312) = A_{n}(213, 321)=1+\binom{n}{2}$;
  		\item[{\upshape (f)}]$A_{n}(123, 321) =\begin{cases} 
  		0 & \text{if\ $n \geq 5$} \\ 
  		n & \text{if\ $n=1$ or\ $n=2$} \\ 
  		4 & \text{if\ $n=3$ or\ $n=4$. }
  		\end{cases} $
  	\end{itemize}
  \end{thm}
  
  In order to determine the distribution of the  statistics
  over $S_{n}(\tau, \rho)$, for every $\tau, \rho \in S_3$, based on the properties of the g.f.'s discussed above, out of all possible  $15$ pairs it is sufficient to examine the distributions of the  statistics over the first pair in each of (a)--(e) in Theorem~\ref{lemma-Sim-class} since the case of $(123,321)$-avoiding permutations is trivial.

This paper is organized as follows. In Section~\ref{distr-sec}, we derive all our distribution results that are summarized in Tables~\ref{tab1} and~\ref{tab2}, where one can find references to the general results and to the formulas giving individual distributions of the statistics, respectively. From our enumerative results we note five equidistributions that are proved combinatorially in Section~\ref{equidis-sec} via introduction of two bijective maps $f$ and $g$. Finally, in Section~\ref{open-sec} we provide concluding remarks.
 
 \begin{table}[htbp]
 	\centering
 	\begin{tabular}{|c|c|c|}
 		\hline
 		&$(\asc,\des,\lrmax,\lrmin,\rlmax,\rlmin)$  &  $(\asc,\des,\MNA,\MND)$\\
 		\hline
 		$S_n(123, 132)$ & Theorem~\ref{thm-F-123, 132}&  Theorem~\ref{gf3-(123, 132)} \\
 		\hline
 		$S_n(132, 321)$&Theorem~\ref{thm-F-132, 321}&Theorem~\ref{gf3-(132, 321)}  \\
 		\hline
 		$S_n(231, 312)$&Theorem~\ref{thm-F-231, 312}& Theorem~\ref{gf3-(231, 312)} \\
 		\hline
 		$S_n(213, 231)$&Theorem~\ref{thm-F-213, 231}& Theorem~\ref{gf3-(213, 231)} \\
 		\hline
 		$S_n(213, 312)$& Theorem~\ref{thm-F-213, 312} & Theorem~\ref{gf3-(213, 312)} \\
 		\hline
 	\end{tabular}
 	\caption{G.f.'s for joint distributions of the statistics over $S_{n}(\tau,\rho)$}
 	\label{tab1}
 \end{table}  
 
  \begin{table}[htbp]
 	\centering
 	\begin{tabular}{|c|c|c|c|c|c|c|c|c|}
 		\hline
 		& $\asc$  &  $\des$ &  $\lrmax$ &  $\rlmax$ & $\lrmin$ & $\rlmin$ & $\MNA$  & $\MND$ \\
 		\hline
 		$S_n(123, 132)$ &  (\ref{123, 132-asc})&  (\ref{123, 132-des}) &  (\ref{123, 132-lrmax}) &  (\ref{123, 132-rlrmax}) & (\ref{123, 132-lrmin}) &  (\ref{123, 132-rlmin}) &  (\ref{cor-gf2-S_{n}(123, 132)MNA}) &  (\ref{cor-gf2-S_{n}(123, 132)MND})  \\
 		\hline
 		$S_n(132, 321)$ &  (\ref{132, 321-asc})&  (\ref{132, 321-des}) &  (\ref{132, 321-lrmax}) &  (\ref{132, 321-rlrmax}) &  (\ref{132, 321-lrmin}) &  (\ref{132, 321-rlmin}) &  (\ref{cor-gf2-S_{n}(132, 321)MNA}) &  (\ref{cor-gf2-S_{n}(132, 321)MND})  \\
 		\hline
 		$S_n(231, 312)$ &  (\ref{231, 312-asc})&  (\ref{231, 312-des}) &  (\ref{231, 312-lrmax}) &  (\ref{231, 312-rlrmax}) &  (\ref{231, 312-lrmin}) &  (\ref{231, 312-rlmin}) &  (\ref{cor-gf2-S_{n}(231, 312)MNA}) &  (\ref{cor-gf2-S_{n}(231, 312)MND}) \\
 		\hline
 		$S_n(213, 231)$ &  (\ref{213, 231-asc})&  (\ref{213, 231-des}) &  (\ref{213, 231-lrmax}) &  (\ref{213, 231-rlrmax}) &  (\ref{213, 231-lrmin}) &  (\ref{213, 231-rlmin}) &  (\ref{cor-gf2-S_{n}(213, 231)MNA}) &  (\ref{cor-gf2-S_{n}(213, 231)MND}) \\
 		\hline
 		$S_n(213, 312)$ &  (\ref{213, 312-asc})&  (\ref{213, 312-des}) &  (\ref{213, 312-lrmax}) &  (\ref{213, 312-rlrmax}) &  (\ref{213, 312-lrmin}) &  (\ref{213, 312-rlmin}) &  (\ref{cor-gf2-S_{n}(213, 312)MNA}) &  (\ref{cor-gf2-S_{n}(213, 312)MND}) \\
 		\hline
 	\end{tabular}
 	\caption{G.f.'s for individual distributions of the statistics over $S_{n}(\tau,\rho)$}
 	\label{tab2}
 \end{table}  
  
\section{Distributions  over $S_n(\tau, \rho)$}\label{distr-sec}

   In this section, we find joint distribution of the seven classical statistics across the five types of arrangements in  Section~\ref{sec1}. Furthermore, we find joint distribution of two more statistics: the maximum number of non-overlapping descents $(\MND)$ and the maximum number of non-overlapping ascents  $(\MNA)$ over the same set of permutations. 
   
   Given permutations $\alpha\in S_a$ and $\beta\in S_b$, let $\alpha \oplus\beta\in S_{a+b}$  denote the direct
   sum of $\alpha$ and $\beta$ and let $\alpha \ominus\beta\in S_{a+b}$ denote the skew-sum of $\alpha$ and $\beta$, defined
   as follows in \cite{Bukata2019}:
   $$\alpha \oplus\beta=\begin{cases} 
   \alpha(i), & 1\le i \le a;\\ 
   a+\beta(i-a), & a+1\le i \le a+b. \\ 
   \end{cases} $$
   
   $$\alpha \ominus\beta=\begin{cases} 
   \alpha(i)+b, & 1\le i \le a;\\ 
   \beta(i-a), & a+1\le i \le a+b. \\ 
   \end{cases} $$
   For example, for $\alpha =123\in S_{3}$ and $\beta =4132\in S_{4}$, $\alpha \oplus\beta=1237465$ and $\alpha \ominus\beta=5674132$.

\subsection{Permutations in $S_{n}(123, 132)$}
We first describe the structure of a $(123, 132)$-avoiding permutation. 
Let $\pi=\pi_1\cdots\pi_n \in S_n(123,132)$. If $\pi_{k}=n, 1<k\le n$, then $\pi_{1}> \pi_{2}>\cdots>\pi_{k-1}$   in order to avoid 123. On the other hand, in order to avoid 132,  $\pi_{i}>  n - k$ if $i < k$. Hence,  $\pi_{i}=n-i$  for $1\le i \le k -1$, while $\pi_{k+1}\pi_{k+2}\cdots \pi_{n}$ must be a (123, 132)-avoiding permutation in $S_{n-k}$. So  
$\pi=(\alpha\oplus 1)\ominus \beta $, where  $\alpha \in S_{k-1}$ is   a decreasing permutation and 
$\beta \in  S_{n-k}$ is a (123,132)-avoiding permutation, and we use the structure of $\pi$ to prove the following theorems.

\begin{thm}\label{gf3-(123, 132)}
	For $S_{n}(123, 132)$, we have
	\begin{equation}\label{123-132-MND-MNA}
		G_{(123, 132)}(x, p, q, y, z)=\frac{A}{1 - 2 q^2 x^2 z - 
			p q x^2 y z - 2 p q^2 x^3 y z + q^4 x^4 z^2 - p q^3 x^4 y z^2},
	\end{equation} where
		\begin{align*}
	&A=1 + x + p x^2 y + q x^2 z - 2 q^2 x^2 z - q^2 x^3 z - p q x^2 y z + 
	2 p q x^3 y z - 2 p q^2 x^3 y z -\\& q^3 x^4 z^2 + q^4 x^4 z^2 + 
	p q^2 x^4 y z^2 - p q^3 x^4 y z^2.
	\end{align*}
\end{thm}
\begin{proof}
	
Let $\pi=\pi_1\cdots\pi_n \in S_n(123,132)$.	  
If $n=0$, it contributes 1 to $G_{(123, 132)}(x, p, q, y, z)$. For $n\ge 1$, we consider three cases based on where the element 
$n$ appears in $\pi$.
	 \begin{itemize}
	 	\item[{\upshape (a)}] If $\pi_{1}=n$, we let the g.f.\ for these permutations be 
	 	$$g_{(123, 132)}(x, p, q, y, z) := \sum_{n\geq 1}\ \sum_{\substack{\pi \in S_{n}(123, 132)\\\pi_{1}=n}}            x^np^{\asc(\pi)}q^{\des(\pi)}y^{\MNA(\pi)}z^{\MND(\pi)}.  $$
	 	\item[{\upshape (b)}] Suppose $\pi_{k}=n$, where $k=2i,i\ge 1$. In this case, $\pi=(\alpha\oplus 1)\ominus \beta $, where  $\alpha \in S_{2i-1}$ is a decreasing permutation with  $i-1$ non-overlapping descents and $2i-2$ descents, the corresponding  g.f.\ is 
	 	$$\sum_{i\ge 1}    x^{2i-1}q^{2i-2}z^{i-1}=\frac{x}{1 - x^2zq^2},$$ and 
	 	$1\ominus \beta $  is a (123,132)-avoiding permutation in $S_{n-2i+1}$. Because  the first element of the permutation $1\ominus \beta $ is the maximum, the corresponding g.f.\ is $g_{(123, 132)}(x, p, q, y, z)$. Additionally, $\pi_{k-1}<\pi_{k}=n$, and $\pi_{k-1}\pi_{k}$ contributes to $\MNA$ giving an extra factor of $yp$. In conclusion, the g.f.\ for permutations in case (b) is $$g_{(123, 132)}(x, p, q, y, z)\frac{xyp}{1 - x^2zq^2}.$$
	 	\item[{\upshape (c)}] Suppose $\pi_{k}=n$, where $k=2i+1,i\ge 1$. In this case, $\pi=(\alpha\oplus 1)\ominus \beta $, where  $\alpha \in S_{2i}$ is a decreasing permutation with $i$ non-overlapping descents and $2i-1$ descents, the corresponding  g.f.\ is 
	 	$$\sum_{i\ge 1}    x^{2i}z^{i}q^{2i-1}=\frac{x^2zq}{1 - x^2zq^2},$$ and 
	 	$1\ominus \beta $  is a (123,132)-avoiding permutation in $S_{n-2i}$. Using similar considerations as those in case (b), the g.f.\ for permutations in case (c) is $$g_{(123, 132)}(x, p, q, y, z)\frac{x^2zypq}{1 - x^2zq^2}.$$
	 \end{itemize}
	 Combining cases (a)--(c), we have
	\begin{eqnarray}\label{$S_{n}(123, 132),A(x,y,z)$}
	&&G_{(123, 132)}(x, p, q, y, z) =
	1 + g_{(123, 132)}(x, p, q, y, z) +
	 g_{(123, 132)}(x, p, q, y, z)\frac{xyp}{1 - x^2zq^2} +\notag \\&& g_{(123, 132)}(x, p, q, y, z)\frac{x^2zypq}{1 - x^2zq^2}.
	\end{eqnarray} 
	Next, we compute $g_{(123, 132)}(x, p, q, y, z)$ similarly to the derivation of $G_{(123, 132)}(x, p, q, y, z)$. If $1\le n\le 2$,  the corresponding g.f.\ is $x+x^2zq$. Next, we distinguish three cases($n\ge3$):
	\begin{itemize}
		\item[{\upshape (d)}] If $\pi_{2}=n-1$ then $\pi_1\pi_2=n(n-1)$  contributes to  $\MND$ that is independent from the count of MND in $\pi_3\cdots\pi_n$, which can be any non-empty permutation in $S_{n-2}(123,132)$. Note that $\pi_2>\pi_3$ contributes to a descent, so the corresponding g.f.\ in this case is $x^2zq^2(G_{(123, 132)}(x, p, q, y, z)-1)$.
		\item[{\upshape (e)}] Suppose $\pi_{m}=n-1$, where $m=2i,i\ge 2$. In this case, $\alpha=\emptyset,\beta=\gamma\ominus 1\ominus \zeta$, so $\pi=1\ominus\gamma\ominus 1\ominus \zeta$, where  $1\ominus\gamma \in S_{2i-1}$ is a decreasing permutation with $i-1$ non-overlapping descents and $2i-2$ descents, and the corresponding  g.f.\ is 
		$$\sum_{i\ge 2}    x^{2i-1}z^{i-1}q^{2i-2}=\frac{x^3zq^2}{1 - x^2zq^2}.$$
		Also,  the permutation
		$1\ominus \zeta $  is in $S_{n-2i+1}(123,132)$ where $\zeta \in  S_{n-2i}$. Because  the first element of  $1\ominus \zeta $ is the maximum, the corresponding g.f. is $g_{(123, 132)}(x, p, q, y, z)$. Moreover, $\pi_{m-1}<\pi_{m}=n-1$, so $\pi_{m-1}\pi_{m}$ forms an extra non-overlapping ascent and ascent. To summarize,
		the corresponding  g.f. for permutations in case (e) is $$g_{(123, 132)}(x, p, q, y, z)\frac{x^3yzpq^2}{1 -x^2zq^2}.$$
		\item[{\upshape (f)}] Suppose $\pi_{m}=n-1$, where $m=2i+1,i\ge 1$.
		In this situation, $\pi=1\ominus\gamma\ominus 1\ominus \zeta $, where  $1\ominus\gamma \in S_{2i}$ is a decreasing permutation contributing $i$ non-overlapping descents and $2i-1$ descents. The g.f. for $1\ominus\gamma \in S_{2i}$ is $$\sum_{i\ge 1}    x^{2i}z^{i}q^{2i-1}=\frac{x^2zq}{1 - x^2zq^2}.$$  In conclusion, the g.f. for the permutations in case (f) is $$g_{(123, 132)}(x, p, q, y, z)\frac{x^2zqyp}{1 - x^2zq^2}.$$
	\end{itemize}
    Summarizing (d)--(f) we obtain 
	\begin{eqnarray}\label{$S_{n}(123, 132),B(x,y,z)$}
	&& g_{(123, 132)}(x, p, q, y, z) = 	x +x^2zq+x^2zq^2(G_{(123, 132)}(x, p, q, y, z)-1)+\notag\\
	&&  g_{(123, 132)}(x, p, q, y, z)\frac{x^3yzpq^2}{1 -x^2zq^2}+ 
	g_{(123, 132)}(x, p, q, y, z)\frac{x^2zqyp}{1 - x^2zq^2}.
	\end{eqnarray} 
	By simultaneously solving \eqref{$S_{n}(123, 132),A(x,y,z)$} and \eqref{$S_{n}(123, 132),B(x,y,z)$}, we obtain (\ref{123-132-MND-MNA}).	
\end{proof}

\begin{cor}Setting three out of the four variables $y$, $z$, $p$ and $q$ equal to one individually in~\eqref{123-132-MND-MNA}, we obtain single distributions of  $\asc$, $\des$, $\MNA$ and $\MND$ over $S_{n}(123, 132)$:
	\begin{eqnarray}
	\sum_{n\geq 0}\ \sum_{\pi \in S_{n}(123, 132)}x^np^{\asc(\pi)}&=&\frac{1 - x}{1 - 2 x + x^2 - p x^2};\label{123, 132-asc}\\
\sum_{n\geq 0}\ \sum_{\pi \in S_{n}(123, 132)}x^nq^{\des(\pi)}&=&\frac{1 + x - 2 q x + x^2 - 2 q x^2 + q^2 x^2}{1 - 2 q x - q x^2 + 
	q^2 x^2};\label{123, 132-des} \\
	\sum_{n\geq 0}\ \sum_{\pi \in S_{n}(123, 132)}x^ny^{\MNA(\pi)}&=&\frac{1 - x}{1 - 2 x + x^2 - x^2 y};\label{cor-gf2-S_{n}(123, 132)MNA}\\
	\sum_{n\geq 0}\ \sum_{\pi \in S_{n}(123, 132)}x^nz^{\MND(\pi)}&=&\frac{1 + x + x^2 - 2 x^2 z - x^3 z}{1 - 3 x^2 z - 2 x^3 z}.\label{cor-gf2-S_{n}(123, 132)MND}
	\end{eqnarray}
\end{cor}

\begin{rem}
The distributions in (\ref{123, 132-asc}) and~(\ref{cor-gf2-S_{n}(123, 132)MNA}) are the same because in $123$-avoiding permutations $\asc=\MNA$.
\end{rem}

\begin{thm}\label{thm-F-123, 132}
	For $S_{n}(123, 132)$, we have
	\begin{equation}\label{F-123, 132}
	\begin{aligned}
	&F_{(123, 132)}(x, p, q, u, v, s, t) =\\
	&\frac{1 + q^2 s^2 v x^2 + s t u v x (1 + p t u x) - 
		q s x (1 + p u v^2 x^2 s t(-1 + t)(-1 + u) + 
		v (1 + p x + s t u x))}{1 + q^2 s^2 v x^2 - 
		q s x (1 + v + p v x)}. 
	\end{aligned}
	\end{equation}
	
\end{thm}
\begin{proof}
		For $\pi=\pi_1\cdots\pi_n \in S_{n}(123, 132)$, if $n=0$, it will give 1 to $F_{(123, 132)}(x, p, q, u, v, s, t)$. Let $n\ge 1$, we consider the following cases.
		\begin{itemize}
	\item If $\pi _1=n$, the element $n$  is the only left-to-right maximum, a left-to-right minimum and a right-to-left maximum, and $\pi_1\pi_2$ is a descent.  So $\lrmax(\pi)=1$ and the g.f. of  permutations  with $\pi _1=n$ is given by $xquvs(F_{(123, 132)}(x, p, q, 1, v, s, t) - 1)+xuvst$, where the element $n$ gives a factor of $xquvs$ (multiplied by the g.f. of all non-empty permutations with the value of $\lrmax$ not taken into account) and the term $xuvst$ corresponds to the permutation of length 1.
	\item If $\pi _n=n$, then $\pi=(n-1)(n-2)\cdots 1n=(\alpha\oplus 1)\ominus \beta $, where  $\beta$ is the empty permutation. The g.f. for the decreasing permutation $\alpha$ is  $$\sum_{i\ge 1}x^iq^{i-1}us^{i}t=\frac{usxt}{1 - xsq}.$$ So, the g.f. for permutations in this case is $$xpuvt\sum_{i\ge 1}x^iq^{i-1}us^{i}t=\frac{u^2sx^2pvt^2}{1 - xsq},$$ where the element $n$ gives a factor of $xpuvt$.
	\item If $\pi_{k}=n, 1<k\le n$, we have $\pi_{1}> \pi_{2}>\cdots>\pi_{k-1}$ and $\lrmax(\pi)=2$. Then $\pi=(\alpha\oplus 1)\ominus \beta $, where any non-empty permutation in $S_{n-k}(123,132)$ is possible for $\beta$.  The g.f. for  $\alpha\oplus 1$ is $$xpquv\sum_{i\ge 1}x^iq^{i-1}us^{i}=\frac{u^2sx^2pqv}{1 - xsq},$$ where the maximum element $n$ gives a factor of $xpquv$. So the g.f. in this case is $$(F_{(123, 132)}(x, p, q, 1, v, s, t) - 1)\frac{u^2sx^2pqv}{1 - xsq}.$$
		\end{itemize}
	
	Synthesizing the above three conditions yields
	\begin{equation}\label{F1,S_{n}(123, 132)}
	\begin{aligned}
	&F_{(123, 132)}(x, p, q, u, v, s, t) = \\
	&1 + xquvs(F_{(123, 132)}(x, p, q, 1, v, s, t) - 1) + xtuvs +\frac{u^2sx^2pvt^2}{1 - xsq}  +\\ 
	& (F_{(123, 132)}(x, p, q, 1, v, s, t) - 1)\frac{u^2sx^2pqv}{1 - xsq}.
	\end{aligned}
	\end{equation} 
	
	Let $u=1$ in \eqref{F1,S_{n}(123, 132)}, we obtain
	\begin{equation}\label{F2,S_{n}(123, 132)}
	\begin{aligned}
	&F_{(123, 132)}(x, p, q, 1, v, s, t) = \\
	&1 + xtvs + xqvs(F_{(123, 132)}(x, p, q, 1, v, s, t) - 1) + \frac{sx^2pvt^2}{1 - xsq}  + \\
	& (F_{(123, 132)}(x, p, q, 1, v, s, t) - 1)\frac{x^2pqvs}{1 - xsq}. 
	\end{aligned}
	\end{equation} 
	
	By simultaneously solving \eqref{F1,S_{n}(123, 132)} and \eqref{F2,S_{n}(123, 132)}, we obtain the desired result.
\end{proof}
\begin{cor}
	Let $p=q=1$, then setting three out of the four variables $u$, $v$, $s$ and $t$ equal to one individually in~\eqref{F-123, 132}, we obtain single distributions of  $\lrmax$, $\rlmax$, $\lrmin$ and $\rlmin$  over $S_{n}(123, 132)$:
	\begin{eqnarray}
		\sum_{n\geq 0}\ \sum_{\pi \in S_{n}(123, 132)}x^nu^{\lrmax(\pi)}&=&\frac{1 - 2 x + u x - u x^2 + u^2 x^2}{1 - 2 x};\label{123, 132-lrmax}\\
		\sum_{n\geq 0}\ \sum_{\pi \in S_{n}(123, 132)}x^nv^{\rlmax(\pi)}&=&\frac{1 - x }{1 - x-vx};\label{123, 132-rlrmax}\\
		\sum_{n\geq 0}\ \sum_{\pi \in S_{n}(123, 132)}x^ns^{\lrmin(\pi)}&=&\frac{1 - s x}{1 - 2 s x - s x^2 + s^2 x^2};\label{123, 132-lrmin}\\
		\sum_{n\geq 0}\ \sum_{\pi \in S_{n}(123, 132)}x^nt^{\rlmin(\pi)}&=&\frac{1 - 2 x + t x - t x^2 + t^2 x^2}{1 - 2x}.\label{123, 132-rlmin}
	\end{eqnarray}
\end{cor}

\begin{rem}
The distributions in (\ref{123, 132-lrmax}) and~(\ref{123, 132-rlmin}) are the same because the patterns $123$ and $132$ are invariant with respect to the (usual group-theoretic) inverse operation which exchanges the sets of left-to-right maxima and right-to-left minima.
\end{rem}

\subsection{Permutations in $S_{n}(132, 321)$}
We first describe the structure of a $(132, 321)$-avoiding permutation. 
Let $\pi=\pi_1\cdots\pi_n \in S_n(132, 321)$. If $\pi_{1}=n$, then $\pi=n12\cdots (n-1)$.  If $\pi_{k}=n, 1<k <n$, then $\pi_{k+1}< \pi_{k+2}<\cdots<\pi_{n}$   in order to avoid 321; on the other hand, in order to avoid 132, $\pi_{i}=n-k+i$  if $1\le i \le k -1$.
If $\pi_{n}=n$ then $\pi_{1}\pi_{2}\cdots\pi_{n-1}\in S_{n-1}(132, 321)$. So $\pi=(\alpha\oplus 1)\ominus \beta $, where  $\alpha \oplus 1\in S_{k}$ and $\beta \in  S_{n-k}$ are   two increasing  (132, 321)-avoiding permutations. We use the structure of $\pi$ to prove the following theorems.

\begin{thm}\label{gf3-(132, 321)}
	For $S_{n}(132, 321)$, we have
	\begin{equation}\label{132-321-MND-MNA}
	G_{(132, 321)}(x, p, q, y, z)= \frac{A}{(1 - p^2 x^2 y)^3}, 
	\end{equation}
	where
	\begin{eqnarray}
	&& A=1 + x + p x^2 y - 3 p^2 x^2 y - 
	2 p^2 x^3 y - 2 p^3 x^4 y^2 + 3 p^4 x^4 y^2 + p^4 x^5 y^2 +\notag\\&& 
	p^5 x^6 y^3 - p^6 x^6 y^3 + q x^2 z + 3 p q x^3 y z + 
	p^2 q x^4 y z + 2 p^2 q x^4 y^2 z + p^3 q x^5 y^2 z.\notag
	\end{eqnarray}
\end{thm}
\begin{proof}
	Let $\pi=\pi_1\cdots\pi_n \in S_n(132, 321)$. The empty permutation, corresponding to the case of $n=0$ gives the term of 1 in $G_{(132, 321)}(x, p, q, y, z)$. If $\pi\in S_1$,  the corresponding g.f. is $x$. For $n\ge 2$, the permutations are divided into three classes depending on the position of $n$.
	\begin{itemize}
		\item[{\upshape (a)}] If $\pi_{1}=n$ then $\pi=n12\cdots(n-1)$. When $n$ is even, the number of non-overlapping ascents is $(n-2)/2$, and the corresponding g.f. is
		$$\sum_{\substack{i\ge 1}}x^{2i}y^{\frac{2i-2}{2}}zqp^{2i-2}=\frac{x^2qz}{1 - x^2yp^2}.$$
		When $n$ is odd, the number of non-overlapping ascents  is $(n-1)/2$, and the corresponding g.f. is
		$$\sum_{\substack{i\ge 1}}x^{2i+1}y^{i}zp^{2i-1}q=\frac{x^3ypqz}{1 - x^2yp^2}.$$
		\item[{\upshape (b)}]Let $\pi_{k}=n$, where $1<k<n$. In this case,  $\pi=(\alpha\oplus 1)\ominus \beta $, where  $\alpha \oplus 1\in S_{k}$ and $\beta \in  S_{n-k}$ are   two increasing  (132, 321)-avoiding permutations. Aditionately,  $\pi_{k}=n>\pi_{k+1}$, and $\pi_{k}\pi_{k+1}$ contributes to $\MND$ giving an extra factor of $z$.  Using similar considerations as those in case (a), the g.f. for permutations in case (b) is $$\frac{z(x^2ypq + x^3yp^2q)(x + x^2yp)}{(1 - x^2yp^2)^2}.$$
		\item[{\upshape (c)}] If $\pi_{n}=n$, we let the g.f. for these permutations be 
		$$g_{(132, 321)}(x, p, q, y, z) := \sum_{n\geq 2}\ \sum_{\substack{\pi \in S_{n}(132, 321)\\\pi_{n}=n}}            x^np^{\asc(\pi)}q^{\des(\pi)}y^{\MNA(\pi)}z^{\MND(\pi)}.  $$
	\end{itemize}
	Combining cases  (a)--(c), we have
	\begin{eqnarray}\label{$S_{n}(132, 321),A(x,y,z)$}
	&&G_{(132, 321)}(x, p, q, y, z) =1 + x +\frac{x^2zq + x^3yzpq}{1 - x^2yp^2}
     +\\&& \frac{z(x^2ypq + x^3yp^2q)(x + x^2yp)}{(1 - x^2yp^2)^2} + g_{(132, 321)}(x, p, q, y, z).  \notag
	\end{eqnarray} 
		Next, we evaluate  $g_{(132, 321)}(x, p, q, y, z)$: 
	\begin{itemize}
		\item[{\upshape (d)}] If $\pi\in S_2$, then $\pi_1\pi_2=12$  and the corresponding g.f. is $x^2yp$.  
		\item[{\upshape (e)}] If $\pi_{1}=n-1$ then $\pi=(n-1)12\cdots n$. Using similar considerations as those in case (a), the g.f. for permutations in case (e) is $$\frac{x^3yzpq + x^4yzp^2q}{1 - x^2yp^2}.$$
		\item[{\upshape (f)}] If $\pi_{m}=n-1$, where $1<m<n-1$, then $\pi=((\gamma\oplus 1)\ominus \zeta)\oplus 1$, where  $\alpha=(\gamma\oplus 1)\ominus \zeta\in S_{n-1}$
		 and $\beta$ is the empty permutation. $\gamma\oplus 1\in S_{m}$ and $\zeta\oplus 1\in S_{n-m}$ are   two increasing  (132, 321)-avoiding permutations. Using similar considerations as those in case (a), the g.f. for permutations in case (f) is $$\frac{(x^2ypq + x^3yp^q)^2zq}{(1 - x^2yp^2)^2}.$$	
		\item[{\upshape (g)}] If $\pi_{n-1}=n-1$ then $\pi_{n-1}\pi_{n}=(n-1)n$ contributes to $\MNA$ giving an extra factor of $x^2yp$. Note that $\pi_{n-2}<\pi_{n-1}=(n-1)$ and any non-empty permutation in $S_{n-2}(132, 321)$ is possible for $\pi_1\cdots\pi_{n-2}$.  The g.f. in case (g) is $x^2yp^2(G_{(132, 321)}(x, p, q, y, z) - 1)$.
	\end{itemize}
	Taking into account  cases (d)--(g), we have
	\begin{eqnarray}\label{$S_{n}(132, 321),B(x,y,z)$}
	&&g_{(132, 321)}(x, p, q, y, z)=x^2yp + \frac{x^3yzpq + x^4yzp^2q}{1 - x^2yp^2}+\\&& 
	\frac{(x^2ypq + x^3yp^q)^2zq}{(1 - x^2yp^2)^2} + x^2yp^2(G_{(132, 321)}(x, p, q, y, z) - 1).\notag
	\end{eqnarray} 
	Solving equations~\eqref{$S_{n}(132, 321),A(x,y,z)$} and \eqref{$S_{n}(132, 321),B(x,y,z)$} simultaneously, we  obtain the desired result (\ref{132-321-MND-MNA}).
\end{proof}

\begin{cor}	Setting three out of the four variables $y$, $z$, $p$ and $q$ equal to one respectively in~\eqref{132-321-MND-MNA}, we obtain single distributions of  $\asc$, $\des$, $\MNA$ and $\MND$ over $S_{n}(132, 321)$:
	\begin{eqnarray}
	\sum_{n\geq 0}\ \sum_{\pi \in S_{n}(132, 321)}x^np^{\asc(\pi)}&=&\frac{1 + x -3 p x + x^2 - 2 p x^2 + 3 p^2 x^2 + p^2 x^3 -p^3 x^3}{(1 - p x)^3};\label{132, 321-asc}\\
	\sum_{n\geq 0}\ \sum_{\pi \in S_{n}(132, 321)}x^nq^{\des(\pi)}&=&\frac{1 - 2 x + x^2 + q x^2}{(1 - x)^3};\label{132, 321-des}\\
	\sum_{n\geq 0}\ \sum_{\pi \in S_{n}(132, 321)}x^ny^{\MNA(\pi)}&=&\frac{1 + x + x^2 - 2 x^2 y + x^3 y + x^4 y + 3 x^4 y^2 + 2 x^5 y^2}{(1-x^2 y)^3};\label{cor-gf2-S_{n}(132, 321)MNA}\\
	\sum_{n\geq 0}\ \sum_{\pi \in S_{n}(132, 321)}x^nz^{\MND(\pi)}&=&\frac{1 - 2 x + x^2 + x^2 z}{(1 - x)^3}.\label{cor-gf2-S_{n}(132, 321)MND}	
	\end{eqnarray}
\end{cor}

\begin{rem}
The distributions in (\ref{132, 321-des}) and~(\ref{cor-gf2-S_{n}(132, 321)MND}) are the same because in $321$-avoiding permutations $\des=\MND$.
\end{rem}

\begin{thm}\label{thm-F-132, 321}
	For $S_{n}(132, 321)$, we have
	\begin{equation}\label{F-132, 321}
	F_{(132, 321)}(x, p, q, u, v, s, t)=\frac{A}{(1 - p t x) (1- p u x) (1 - p t u x)}
	\end{equation}
	where
	\begin{align*}
	&A=1 + s t u v x + q s^2 t u v^2 x^2 - p^3 t^2 u^2 x^3 + 
	p^2 t u x^2 (1 + t + u + s t u v x) - \\
	&p x (u + s t^2 u v x (1 + q s u (-1 + v) x) + 
	t (1 + u + s u^2 v x).
	\end{align*}
\end{thm}
\begin{proof}
	Let $\pi=\pi_1\cdots\pi_n \in S_n(132, 321)$. If $n=0$, we get the term of 1 in $A_{(132, 321)}(x,y,z)$. If $\pi\in S_1$,  the corresponding g.f. is $xuvst$. For $n\ge 2$, we consider the following cases.
	\begin{itemize}
		\item If $\pi _1=n$ then $\pi=n12\cdots(n-1)$. The element $n$  is the only left-to-right maximum, a left-to-right minimum and   a right-to-left maximum, and $\pi_1\pi_2$ is a descent. So $\lrmax(\pi)=1$ and the g.f. of  permutations  with $\pi _1=n$ is given by 
		$$xquvs\sum_{i\ge 2}x^{i-1}p^{i-2}vst^{i-1}=\frac{x^2quv^2s^2t}{1 - xpt},$$ where the element $n$ gives a factor of $xquvs$.
		\item If $\pi _n=n$ then $\rlmax(\pi)=1$.
		 Any non-empty permutation in $S_{n-1}(132, 321)$ is possible for $\pi_{1} \pi_{2}\cdots\pi_{n-1}$ and we do not need to consider right-to-left maxima. So the g.f.  in this case is $$xpuvt (F_{(132, 321)}(x, p, q, u, 1, s, t) - 1).$$
		\item If $\pi_{k}=n$, $1<k<n$, then $\pi=(\alpha\oplus 1)\ominus \beta $, where  $\alpha \in S_{k-1}$ and $\beta \in  S_{n-k}$ are   two increasing  (132, 321)-avoiding permutations. The g.f. for the  permutation $\alpha \in S_{k-1}$ is %denoted as
		 $$\sum_{i\ge 1}x^ip^{i}u^is=\frac{xpus}{1 - xpu}$$  (note that  $\pi_{k-1}\pi_{k}$ is an ascent).
		The g.f. for $1\ominus \beta \in  S_{n-k+1}$ is $$xquv\sum_{i\ge 2}x^{i-1}p^{i-2}vst^{i-1}=\frac{x^2quv^2st}{1 - xpt},$$ where  the element $n$ gives the factor of $xquv$. So the g.f. in this case  is $$\frac{x^3pqu^2v^2s^2t}{(1 - xpt)(1 - xpu)}.$$
	\end{itemize}

		Taking into account all the cases, we  conclude that
		\begin{align}\label{F1,S_{n}(132, 321)}
			&F_{(132, 321)}(x, p, q, u, v, s,t)=1 + xtuvs  +\frac{x^2quv^2s^2t}{1 - xpt}
		+ \\&\frac{x^3pqu^2v^2s^2t}{(1 - xpt)(1 - xpu)}+ xpuvt (F_{(132, 321)}(x, p, q, u, 1, s, t) - 1). \notag
		\end{align}
		Let $v=1$ in \eqref{F1,S_{n}(132, 321)}, we get 
	\begin{align}\label{F2,S_{n}(132, 321)}
	&F_{(132, 321)}(x, p, q, u, 1, s, t) = 
	1 + xtus + \frac{x^2qus^2t}{1 - xpt} +\\& \frac{xpusx^2qust}{(1 - xpt)(1 - xpu)} + 
	xput (F_{(132, 321)}(x, p, q, u, 1, s, t) - 1).\notag
	\end{align} 
		By simultaneously solving \eqref{F1,S_{n}(132, 321)} and \eqref{F2,S_{n}(132, 321)}, we obtain the desired result.
\end{proof}
\begin{cor}
	Let $p=q=1$, then setting three out of the four variables $u$, $v$, $s$ and $t$ equal to one individually in~\eqref{F-132, 321}, we obtain single distributions of  $\lrmax$, $\rlmax$, $\lrmin$ and $\rlmin$  over $S_{n}(132, 321)$:
	\begin{eqnarray}
	\sum_{n\geq 0}\ \sum_{\pi \in S_{n}(132, 321)}x^nu^{\lrmax(\pi)}&=&\frac{1 - x - u x+ 2 u x^2}{(1 - x) (1 - u x)^2};\label{132, 321-lrmax}\\
	\sum_{n\geq 0}\ \sum_{\pi \in S_{n}(132, 321)}x^nv^{\rlmax(\pi)}&=&\frac{1 - 3 x + v x + 3 x^2 - 2 v x^2 +v^2 x^2 - x^3 + 2 v x^3 - v^2 x^3 }{(1 - x)^3};\label{132, 321-rlrmax}\\
	\sum_{n\geq 0}\ \sum_{\pi \in S_{n}(132, 321)}x^ns^{\lrmin(\pi)}&=&\frac{1 - 3 x + s x + 3 x^2 - 2 s x^2 + s^2 x^2 - x^3 +s x^3}{(1 - x)^3};\label{132, 321-lrmin}\\
	\sum_{n\geq 0}\ \sum_{\pi \in S_{n}(132, 321)}x^nt^{\rlmin(\pi)}&=&\frac{1 - x - t x + 2 t x^2}{(1 - x) (1 - t x)^2}.\label{132, 321-rlmin}
	\end{eqnarray}
\end{cor}

\begin{rem}
The distributions in (\ref{132, 321-lrmax}) and~(\ref{132, 321-rlmin}) are the same because the patterns $132$ and $321$ are invariant with respect to the inverse operation which exchanges the sets of left-to-right maxima and right-to-left minima.
\end{rem}

\subsection{Permutations in $S_{n}(231, 312)$}
We first describe the structure of a $(231, 312)$-avoiding permutation. 
Let $\pi=\pi_1\cdots\pi_n  \in S_n(231, 312)$. If $\pi_{1}=n$ then $\pi=n(n-1)\cdots 21$. If $\pi_{k}=n, 1<k <n$, then $\pi_{k+1}> \pi_{k+2}>\cdots>\pi_{n}$   in order to avoid 312. On the other hand, in order to avoid 231, $\pi_{i}=n+k-i$  if $k+1\le i \le n$, $\pi_{1}\pi_{2} \cdots \pi_{k-1}$ must be a  permutation in $S_{k-1}(231, 312)$. If $\pi_{n}=n$, $\pi_{1} \pi_{2}\cdots \pi_{n-1}$ must be a permutation in $S_{n-1}(231, 312)$. Namely, for $\pi \in S_{n}(231, 312)$, its structure is 
$\pi=\alpha\oplus( 1\ominus \beta) $, where  $\alpha \in S_{k-1}(231, 312)$ and $1\ominus \beta \in  S_{n-k+1}$ is a  decreasing  (231, 312)-avoiding permutation.  We use the structure of $\pi$ to prove the following theorems.

\begin{thm}\label{gf3-(231, 312)}
	For $S_{n}(231, 312)$, we have
	\begin{align}\label{(231, 312)-G_MND-MNA-des-asc}
	&G_{(231, 312)}(x, p, q, y, z) = \\&\frac{1 + x + p x^2 y - p^2 x^2 y + q x^2 z - q^2 x^2 z - p q x^2 y z + 
		p q x^3 y z - p^2 q x^3 y z - p q^2 x^3 y z}{1 - p^2 x^2 y - 
		q^2 x^2 z - p q x^2 y z - p^2 q x^3 y z - p q^2 x^3 y z}. \notag
	\end{align}
\end{thm}
\begin{proof}
	Let $\pi=\pi_1\cdots\pi_n  \in S_n(231, 312)$.	
	If $n\le 1$, we have the term of $1+x$ in $G_{(231, 312)}(x, p, q, y, z)$. For $n\ge 2$, the permutations are divided into three classes depending on the position of $n$.
	\begin{itemize}
		\item[{\upshape (a)}] If $\pi_{1}=n$ then $\pi=n(n-1)\cdots 2 1$. When $n$ is even, $\pi$ has  $n/2$ non-overlapping descents and $n-1$ descents. The corresponding g.f. is
		$$\sum_{i\ge 1}x^{2i}z^{i}q^{2i-1}=\frac{x^2zq}{1 - x^2q^2z}.$$
		When $n$ is odd, $\pi$ has $(n-1)/2$ non-overlapping descents and $n-1$ descents. The corresponding g.f. is
		$$\sum_{i\ge 1}x^{2i+1}z^{i}q^{2i}=\frac{x^3zq^2}{1 - x^2zq^2}.$$ 
		\item[{\upshape (b)}] If $\pi_{n}=n$, we let the g.f. for these permutations be 
		$$g_{(231, 312)}(x, p, q, y, z) :=\sum_{n\geq 2}\  \sum_{\substack{\pi \in S_{n}(132, 321)\\\pi_{n}=n}}            x^np^{\asc(\pi)}q^{\des(\pi)}y^{\MNA(\pi)}z^{\MND(\pi)}.  $$
		\item[{\upshape (c)}] If $\pi_{k}=n$, $1<k<n$, then $\pi=\alpha\oplus( 1\ominus \beta) $, where  $\alpha \in S_{k-1}(231, 312)$ and $1\ominus \beta \in  S_{n-k+1}$ is a  decreasing  (231, 312)-avoiding permutation. For $\alpha\oplus 1$, the g.f. is $g_{(231, 312)}(x, p, q, y, z)$. For $\beta \in  S_{n-k+1}$, similarly to case (a), we see that the corresponding  g.f. is $(xzq + x^2zq^2)/(1 - x^2zq^2)$ (note that  $\pi_{k}\pi_{k+1}$ contributes to $\MND$).
	\end{itemize}
	Combining cases (a)--(c), we have
	\begin{eqnarray}\label{G-1(231, 312)}
	&&	G_{(231, 312)}(x, p, q, y, z) =   1 + x +\frac{x^2zq + x^3zq^2}{1 - x^2zq^2}+\\&& g_{(231, 312)}(x, p, q, y, z)\frac{xzq + x^2zq^2}{1 - x^2zq^2} + g_{(231, 312)}(x, p, q, y, z). \notag
	\end{eqnarray} 
	
	Next we evaluate  $g_{(231, 312)}(x, p, q, y, z)$: 
	\begin{itemize}
		\item[{\upshape (d)}] If $n=2$, the g.f. is $x^2yp$.  
		\item[{\upshape (e)}] If $\pi_{1}=n-1$ then $\pi=(n-1)(n-2)\cdots 1 n$, and the corresponding g.f. is $$xyp\frac{x^2zq + x^3zq^2}{1 - x^2zq^2},$$ where the element $n$ gives a factor of $xyp$.
		\item[{\upshape (f)}]  If $\pi_{m}=n-1$, $1<m<n$, then $\pi=\gamma \oplus( 1\ominus \zeta)\oplus 1 $, where  $\gamma \in S_{m-1}(231, 312)$ and $ \zeta \in  S_{n-m-1}(231, 312)$. For $\gamma \oplus 1$, the  g.f. is $g_{(231, 312)}(x, p, q, y, z)$. For $\zeta\oplus 1\in  S_{n-m+1}$, because the structure is the same as in case (b), we obtain the  g.f. is $(xzq + x^2zq^2)/(1 - x^2zq^2)$ (recall that if $\zeta$ is of odd length, $\pi_m\pi_{m+1}$ will contribute to $\MND$). 
		To summarize,	the  g.f. in case (f) is $$xyp	g_{(231, 312)}(x, p, q, y, z)\frac{xzq + x^2zq^2}{1 - x^2zq^2},$$ where the element $n$ gives the factor of $xyp$.
		\item[{\upshape (g)}] If $\pi_{n-1}=n-1$ then $\pi_{n-1}\pi_{n}=(n-1)n$ contributes to $\MNA$, and it is independent from the count of MNA in $\pi_1\cdots\pi_{n-2}$, which can be any  permutation in $S_{n-2}(231, 312)$. So the corresponding g.f. in this case is  $x^2yp^2(G_{(231, 312)}(x, p, q, y, z) - 1)$.
	\end{itemize}
	
	Combining cases (d)--(g), we have
	\begin{align}\label{G-2(231, 312)}
	&g_{(231, 312)}(x, p, q, y, z) = x^2yp +xyp\frac{x^2zq + x^3zq^2}{1 - x^2zq^2}+ \\&
	xyp	g_{(231, 312)}(x, p, q, y, z)\frac{xzq + x^2zq^2}{1 - x^2zq^2} + x^2yp^2(G_{(231, 312)}(x, p, q, y, z) - 1). \notag
	\end{align} 
	Solving the equations~\eqref{G-1(231, 312)} and \eqref{G-2(231, 312)} simultaneously, we  obtain 	(\ref{(231, 312)-G_MND-MNA-des-asc}).
\end{proof}

\begin{cor}	Setting three out of the four variables $y$, $z$, $p$ and $q$ equal to one respectively in~\eqref{(231, 312)-G_MND-MNA-des-asc}, we obtain single distributions of  $\asc$, $\des$, $\MNA$ and $\MND$ over $S_{n}(231, 312)$:
	\begin{eqnarray}
	\sum_{n\geq 0}\ \sum_{\pi \in S_{n}(231, 312)}x^np^{\asc(\pi)}&=&\frac{1 - p x}{1 - x - p x};\label{231, 312-asc}\\
	\sum_{n\geq 0}\ \sum_{\pi \in S_{n}(231, 312)}x^nq^{\des(\pi)}&=&\frac{1 - q x}{1 - x - q x};\label{231, 312-des}\\
	\sum_{n\geq 0}\ \sum_{\pi \in S_{n}(231, 312)}x^ny^{\MNA(\pi)}&=&\frac{1 - x^2 y}{1 - x - 2 x^2 y};\label{cor-gf2-S_{n}(231, 312)MNA}\\
	\sum_{n\geq 0}\ \sum_{\pi \in S_{n}(231, 312)}x^nz^{\MND(\pi)}&=&\frac{1 - x^2 z}{1 - x - 2 x^2 z}.\label{cor-gf2-S_{n}(231, 312)MND}
	\end{eqnarray}
\end{cor}

\begin{rem} The same distributions in (\ref{231, 312-asc}) and (\ref{231, 312-des}), as well as in (\ref{cor-gf2-S_{n}(231, 312)MNA}) and (\ref{cor-gf2-S_{n}(231, 312)MND}), follow from a more general Theorem~\ref{equidist-thm-1}.
\end{rem}

\begin{thm}\label{thm-F-231, 312}
	For $S_{n}(231, 312)$, we have
	\begin{equation}\label{F-231, 312}
	F_{(231, 312)}(x, p, q, u, v, s, t)=\frac{A}{(1 - 
		q s x) (1 - q x - p t u x) (1 - q v x) (1 - q s v x)}
	\end{equation}
	where
	\begin{align*}
	&A=1 - p t u x + s t u v x + q^4 s^2 v^2 x^4 + 
	q^3 s v x^3 (-1 - v + s (-1 + v (-1 + (-1 + p) t u x))) 
	- \\&
	q x (1 + v - p t u v x + s^2 t u v x (1 + p t u (-1 + v) x) + 
	s (1 + v - p t u x - (-1 + p) t u v x 
	+\\&p t^2 u^2 v x^2 + 
	t u v^2 x (1 - p t u x))) + 
	q^2 x^2 (v + s^2 v (1 + t u (1 - p + v) x) 
	+ \\&
	s (1 + v^2 (1 - (-1 + p) t u x) + v (2 - p t u x))).
	\end{align*}
\end{thm}
\begin{proof}
		Let $\pi=\pi_1\cdots\pi_n \in S_n(231, 312)$. The case of $n=0$ contributes the term 1 to $A_{(231, 312)}(x,y,z)$. If $\pi\in S_1$, the g.f. is $xuvst$. For $n\ge 2$, we consider the following cases.
	\begin{itemize}
		\item If $\pi _1=n$ then $\pi=n(n-1)\cdots 1$. The element $n$  is the only left-to-right maximum, a left-to-right minimum and  a right-to-left maximum, and $\pi_1\pi_2$ is a descent. So $\lrmax(\pi)=1$ and the g.f. of  permutations  with $\pi _1=n$ is given %expressed 
		by 
		$$xquvs\sum_{i\ge 2}x^{i-1}q^{i-2}v^{i-1}s^{i-1}t=\frac{x^2quv^2s^2t}{1 - xqvs},$$ where the element $n$ gives the factor of  $xquvs$.
		\item If $\pi _n=n$, then $\rlmax(\pi)=1$.
		Any non-empty permutation in $S_{n-1}(231, 312)$ is possible for $\pi_{1} \pi_{2}\cdots\pi_{n-1}$ and we do not need to consider right-to-left maxima. Therefore, the g.f. is\\  $xpuvt (F_{(231, 312)}(x, p, q, u, 1, s, t) - 1)$, where the element $n$ gives the factor of $xpuvt$.
		\item  If $\pi_{k}=n$, $1<k<n$, then $\pi=\alpha\oplus( 1\ominus \beta) $, where  $\alpha \in S_{k-1}(231, 312)$ and $1\ominus \beta \in  S_{n-k+1}$ is a  decreasing  (231, 312)-avoiding permutation.  For $\alpha\oplus 1 \in S_{k-1}$, because we do not need to consider right-to-left maxima, the g.f. is $xpquv (F_{(231, 312)}(x, p, q, u, 1, s, t) - 1)$, where the element $n$ gives the factor of $xpquv$. For $\beta$, we have $$\sum_{i\ge 1}x^{i}q^{i-1}v^it=\frac{xvt}{1 - xqv}.$$So the g.f. in this case  is $(F_{(231, 312)}(x, p, q, u, 1, s, t) - 1)\frac{x^2pquv^2t}{1 - xqv}$.
	\end{itemize}

		Taking into account all cases, we obtain
	\begin{align}\label{F1,S_{n}(231, 312)}
	&F_{(231, 312)}(x, p, q, u, v, s, t) = 
	1 + xtuvs + \frac{x^2quv^2s^2t}{1 - xqvs} + \\&(F_{(231, 312)}(x, p, q, u, 1, s, t) - 1)\frac{x^2pquv^2t}{1 - xqv} + xpuvt (F_{(231, 312)}(x, p, q, u, 1, s, t) - 1).  \notag
	\end{align} 
	
	Let $v=1$ in \eqref{F1,S_{n}(231, 312)}, we obtain
	\begin{align}\label{F2,S_{n}(231, 312)}
	&F_{(231, 312)}(x, p, q, u, 1, s, t) =	1 + xtus + \frac{ x^2qus^2t}{1 - xqs}+ \\&
	(F_{(231, 312)}(x, p, q, u, 1, s, t) - 1)\frac{x^2pqut}{1 - xq} + xput (F_{(231, 312)}(x, p, q, u, 1, s, t) - 1).\notag
	\end{align}
	
	By simultaneously solving \eqref{F1,S_{n}(231, 312)} and \eqref{F2,S_{n}(231, 312)}, we obtain the desired result.
\end{proof}
\begin{cor}
	Let $p=q=1$, then setting three out of the four variables $u$, $v$, $s$ and $t$ equal to one individually in~\eqref{F-231, 312}, we obtain single distributions of  $\lrmax$, $\rlmax$, $\lrmin$ and $\rlmin$  over $S_{n}(231, 312)$:
	\begin{eqnarray}
	\sum_{n\geq 0}\ \sum_{\pi \in S_{n}(231, 312)}x^nu^{\lrmax(\pi)}&=&\frac{1 - x}{1 - x - u x};\label{231, 312-lrmax}\\
	\sum_{n\geq 0}\ \sum_{\pi \in S_{n}(231, 312)}x^nv^{\rlmax(\pi)}&=&\frac{1 - 2 x + v x^2}{(1 - 2 x) (1 - v x)};\label{231, 312-rlrmax}\\
	\sum_{n\geq 0}\ \sum_{\pi \in S_{n}(231, 312)}x^ns^{\lrmin(\pi)}&=&\frac{1 - 2 x + s x^2}{(1 - 2 x) (1 - s x)};\label{231, 312-lrmin}\\
	\sum_{n\geq 0}\ \sum_{\pi \in S_{n}(231, 312)}x^nt^{\rlmin(\pi)}&=&\frac{1 - x}{1 - x - t x}.\label{231, 312-rlmin}
	\end{eqnarray}
\end{cor}

\begin{rem}
The distributions in (\ref{231, 312-lrmax}) and~(\ref{231, 312-rlmin}) (resp., (\ref{231, 312-rlrmax}) and~(\ref{231, 312-lrmin})) are the same because the set $S_{n}(231, 312)$ is invariant under the composition of the reverse and complement operations, and applying the composition exchanges the sets of left-to-right maxima and right-to-left minima (resp., right-to-left maxima and left-to-right minima). 
\end{rem}

\subsection{Permutations in $S_{n}(213, 231)$}
We first describe the structure of a $(213, 231)$-avoiding permutation. 
Let $\pi=\pi_1\cdots\pi_n \in S_n(213, 231)$. If $\pi_{1}=n$ then $\pi=n(n-1)\cdots2 1$. If $\pi_{k}=n, 1<k <n$, then $\pi_{1}< \pi_{2}<\cdots<\pi_{k-1}$   in order to avoid 213. On the other hand, in order to avoid 231, $\pi_{i}>\pi_{k-1}$  if $k+1\le i \le n$. 
If $\pi_{n}=n$ then $\pi=1 2\cdots n$. 
So, for $\pi \in S_{n}(213, 231)$, its structure is 
$\pi=\alpha\oplus( 1\ominus \beta) $, where  $\alpha \in S_{k-1}$ is an  increasing  (213, 231)-avoiding permutation and $1\ominus \beta \in  S_{n-k+1}(213, 231)$, and we use the structure of $\pi$ to prove the following theorems.

\begin{thm}\label{gf3-(213, 231)}
	For $S_{n}(213, 231)$, we have
	\begin{align}\label{213, 231-MND-MNA-des-asc}
	&G_{(213, 231)}(x, p, q, y, z) = \\&\frac{1 + x + p x^2 y - p^2 x^2 y + q x^2 z - q^2 x^2 z - p q x^2 y z + 
		p q x^3 y z - p^2 q x^3 y z - p q^2 x^3 y z}{1 - p^2 x^2 y - 
		q^2 x^2 z - p q x^2 y z - p^2 q x^3 y z - p q^2 x^3 y z}. \notag
	\end{align}
\end{thm}
\begin{proof}
	Let $\pi=\pi_1\cdots\pi_n \in S_n(213, 231)$ .
	If $n\le 1$ then $G_{(213, 231)}(x, p, q, y, z)=1+x$. For $n\ge 2$, the permutations are divided into three classes depending on the position of $n$.
	\begin{itemize}
		\item[{\upshape (a)}] If $\pi_{1}=n$, we let the g.f. for these permutations be 
		$$g_{(213, 231)}(x, p, q, y, z) := \sum_{n\geq 2}\ \sum_{\substack{\pi \in S_{n}(213, 231)\\\pi_{1}=n}}            x^np^{\asc(\pi)}q^{\des(\pi)}y^{\MNA(\pi)}z^{\MND(\pi)}.  $$ 
		\item[{\upshape (b)}] If $\pi_{n}=n$ then $\pi=12 \cdots n$. When $n$ is even, $\pi$ has  $n/2$ non-overlapping ascents and $n-1$  ascents, and the corresponding g.f. is
		$$\sum_{i\ge 1}x^{2i}y^{i}p^{2i-1}=\frac{x^2yp}{1 - x^2yp^2}.$$
		When $n$ is odd, $\pi$ has  $(n-1)/2$ non-overlapping ascents and $n-1$ ascents, and the corresponding g.f. is
		$$\sum_{i\ge 1}x^{2i+1}y^{i}p^{2i}=\frac{x^3yp^2}{1 - x^2yp^2}.$$ 
		\item[{\upshape (c)}] If $\pi_{k}=n$, $1<k<n$, then $\pi=\alpha\oplus( 1\ominus \beta) $, where  $\alpha \in S_{k-1}$ is an  increasing  (213, 231)-avoiding permutation and $1\ominus \beta \in  S_{n-k+1}(213, 231)$, whose corresponding  g.f. is $g_{(213, 231)}(x, p, q, y, z)$. For $\alpha \in S_{k-1}$, take into account that if the increasing sequence is of odd length,  $\pi_{k-1}\pi_{k}$ contributes to  $\MNA$ giving an extra factor of $y$.
		To summarize, in this case the  g.f. is $$g_{(213, 231)}(x, p, q, y, z)\frac{xpy + x^2yp^2}{1 - x^2p^2y}.$$
	\end{itemize}
	Combining cases (a)--(c), we obtain
	\begin{eqnarray}\label{$(213, 231)-G-1$}
	&&	G_{(213, 231)}(x, p, q, y, z) =  1 + x +g_{(213, 231)}(x, p, q, y, z)+\notag\\ && g_{(213, 231)}(x, p, q, y, z)\frac{xpy + x^2yp^2}{1 - x^2p^2y}
	+ \frac{x^2yp + x^3yp^2}{1 - x^2p^2y}. 
	\end{eqnarray} 
	Next, we evaluate  $g_{(213, 231)}(x, p, q, y, z)$:  
	\begin{itemize}
		\item[{\upshape (d)}] If $n=2$, the g.f. is $x^2zq$.  
		\item[{\upshape (e)}] If $\pi_{2}=n-1$ then any non-empty permutation in $S_{n}(213, 231)$ is possible for $\pi_3\cdots\pi_n$. The corresponding g.f. is $x^2zq^2(G_{(213, 231)}(x, p, q, y, z) - 1)$, where $\pi_{1}\pi_{2}$ contributes to   $\MND$ giving an extra factor of $x^2zq^2$($\pi_{2}>\pi_{3}$).
		\item[{\upshape (f)}] If $\pi_{n}=n-1$  then $\pi=1\ominus(\gamma\oplus 1 )$, where  $\gamma\oplus 1 \in S_{n-1}$ is an  increasing  (213, 231)-avoiding permutation. In this case,  
		the corresponding  g.f. is $$\frac{x^3yzpq + x^4yzp^2q}{1 - x^2p^2y}.$$
		\item[{\upshape (g)}] 	If $\pi_{m}=n-1$, $2<m<n$, then $\pi=1\ominus(\gamma\oplus( 1\ominus \zeta) )$, where  $\gamma \in S_{m-2}$ is an  increasing  (213, 231)-avoiding permutation and $1\ominus \zeta \in  S_{n-m+1}(213, 231)$, whose corresponding  g.f. is $g_{(213, 231)}(x, p, q, y, z)$. For $1\ominus\gamma \in S_{m-1}$, note that if $\gamma$ contains an odd number of elements, $\pi_{m-1}\pi_{m}$  contributes to  $\MNA$.
		To summarize, in this case the  g.f. is $$g_{(213, 231)}(x, p, q, y, z)\frac{x^2yzpq + x^3yzp^2q}{1 - x^2p^2y},$$ where the element $n$ gives a factor of $xzq$.
	\end{itemize}
	Combining cases (d)--(g), we obtain
	\begin{align}\label{$(213, 231)-G-2$}
	&g_{(213, 231)}(x, p, q, y, z) =x^2zq +x^2zq^2(G_{(213, 231)}(x, p, q, y, z) - 1)+\notag\\&\frac{x^3yzpq + x^4yzp^2q}{1 - x^2p^2y}+g_{(213, 231)}(x, p, q, y, z)\frac{x^2yzpq + x^3yzp^2q}{1 - x^2p^2y} . 
	\end{align} 
	Solving equations~\eqref{$(213, 231)-G-1$} and \eqref{$(213, 231)-G-2$} simultaneously, we  obtain  (\ref{213, 231-MND-MNA-des-asc}).
\end{proof}

From Theorem~\ref{gf3-(213, 231)}  we have the following results. 

\begin{cor}\label{MNA-MND-231-312-213-231}
		Setting three out of the four variables $y$, $z$, $p$ and $q$ equal to one respectively in~\eqref{213, 231-MND-MNA-des-asc}, we obtain single distributions of  $\asc$, $\des$, $\MNA$ and $\MND$ over $S_{n}(213, 231)$:
	\begin{eqnarray}
	\sum_{n\geq 0}\ \sum_{\pi \in S_{n}(213, 231)}x^np^{\asc(\pi)}&=&\frac{1 - p x}{1 - x - p x};\label{213, 231-asc}\\
	\sum_{n\geq 0}\ \sum_{\pi \in S_{n}(213, 231)}x^nq^{\des(\pi)}&=&\frac{1 - q x}{1 - x - q x};\label{213, 231-des}\\	
	\sum_{n\geq 0}\ \sum_{\pi \in S_{n}(213, 231)}x^ny^{\MNA(\pi)}&=&\frac{1 - x^2 y}{1 - x - 2 x^2 y};\label{cor-gf2-S_{n}(213, 231)MNA}\\
	\sum_{n\geq 0}\ \sum_{\pi \in S_{n}(213, 231)}x^nz^{\MND(\pi)}&=&\frac{1 - x^2 z}{1 - x - 2 x^2 z}.\label{cor-gf2-S_{n}(213, 231)MND}	
	\end{eqnarray}
\end{cor}

\begin{rem}
The distributions in (\ref{213, 231-asc}) and~(\ref{213, 231-des}) (resp., (\ref{cor-gf2-S_{n}(213, 231)MNA}) and~(\ref{cor-gf2-S_{n}(213, 231)MND})) are the same because the set $S_{n}(213, 231)$ is invariant under the complement operation, and applying complement exchanges ascents and descents (resp., non-overlapping ascents and non-overlapping descents). 
\end{rem}
 
\begin{thm}\label{thm-F-213, 231}
	For $S_{n}(213, 231)$, we have
	\begin{equation}\label{F-213, 231}
	F_{(213, 231)}(x, p, q, u, v, s, t)=\frac{A}{(1 - p t u x) (1 - p t x - 
		q v x) (1 - q s v x)}
	\end{equation}
	where $A$ is given by 
	\begin{align*}
	&1 - p t x - p t u x - q v x - q s v x + s t u v x + 
	p^2 t^2 u x^2 + p q s t v x^2 + p q t u v x^2 + 
	p q s t u v x^2 - p s t^2 u v x^2 +\\& q^2 s v^2 x^2 - 
	q s t u v^2 x^2 - p^2 q s t^2 u v x^3 - p q^2 s t u v^2 x^3 + 
	p q s^2 t^2 u v^2 x^3 + p q s t^2 u^2 v^2 x^3 - 
	p q s^2 t^2 u^2 v^2 x^3.
	\end{align*}
\end{thm}
\begin{proof}
	Let $\pi=\pi_1\cdots\pi_n \in S_n(213, 231)$. If $n=0$, $\pi$ contributes the term of 1 to $A_{(213, 231)}(x,y,z)$. If $\pi\in S_1$, the g.f. is $xuvst$. For $n\ge 2$, we consider the following cases.
	\begin{itemize}
		\item If $\pi _1=n$  then the element $n$  is the only left-to-right maximum, a left-to-right minimum and a right-to-left maximum, and $\pi_1\pi_2$ is a descent. So $\lrmax(\pi)=1$ and the g.f. of  permutations  with $\pi _1=n$ is given by $xquvs (F_{(213, 231)}(x, p, q, 1, v, s, t) - 1)$, where we used the g.f. of all non-empty permutations with the value of $\lrmax$ not taken into account and the element $n$ gives the factor of $xquvs$.
		\item If $\pi _n=n$ then $\pi=12 \cdots n$. So we have
		$$xpuvt\sum_{i\ge 1}x^{i}p^{i-1}u^ist^i=\frac{x^2pu^2vst^2}{1 - xput} ,$$ where the element $n$ gives the factor of $xpuvt$.
		\item If $\pi_{k}=n, 1<k< n$, then  $\pi=\alpha\oplus( 1\ominus \beta) $, where  $\alpha $ is an  increasing  permutation in $ S_{k-1}(213, 231)$  and $1\ominus \beta \in  S_{n-k+1}(213, 231)$.
		The g.f. for $\alpha \in S_{k-1}$ is  $\frac{xust}{1 - xput}$ and the element $n$ gives a factor of $xpquv$. 
		For the permutation $ \beta \in  S_{n-k}$,
		we do not need to consider left-to-right maxima and  left-to-right minima, so the g.f. is $(F_{(213, 231)}(x, p, q, 1, v, 1,  t) - 1)$.	The g.f. of  permutations  with $\pi_{k}=n, 1<k< n$, is  $$(F_{(213, 231)}(x, p, q, 1,  v,  1, t) - 1)\frac{x^2pqu^2vst}{1 - xput}.$$
	\end{itemize}

Taking into account all cases, we obtain 
	\begin{align}\label{F1,S_{n}(213, 231)}
	&F_{(213, 231)}(x, p, q, u, v, s, t) = 
	1 + xtuvs + xquvs(F_{(213, 231)}(x, p, q,  1, v, s, t) - 1) +\notag\\&(F_{(213, 231)}(x, p, q,  1, v,   1, t) - 1)\frac{x^2pqu^2vst}{1 - xput}
	+ \frac{x^2pu^2vst^2}{1 - xput};
	\end{align} 
	
	If $u=1$ in \eqref{F1,S_{n}(213, 231)}, we  have
	\begin{align}\label{F2,S_{n}(213, 231)}
	&F_{(213, 231)}(x, p, q,  1, v, s, t) = 
	1 + xtvs + xqvs(F_{(213, 231)}(x, p, q,  1, v, s, t) - 1) + \notag\\&
	(F_{(213, 231)}(x, p, q,  1, v,  1, t) - 1)\frac{	x^2pqvst}{1 - xpt}+\frac{	x^2pvst^2}{1 - xpt} ;
	\end{align} 
		
	If $s=1$ in \eqref{F2,S_{n}(213, 231)}, we have
	\begin{align}\label{F3,S_{n}(213, 231)}
	&F_{(213, 231)}(x, p, q,  1, v,  1, t) = 
	1 + xtv + xqv(F_{(213, 231)}(x, p, q,  1, v,  1, t) - 1) + \notag\\&
	(F_{(213, 231)}(x, p, q,  1, v,  1, t) - 1)\frac{x^2pqvt}{1 - xpt}+\frac{	x^2pvt^2}{1 - xpt} .
	\end{align} 
	
	By simultaneously solving \eqref{F1,S_{n}(213, 231)},\eqref{F2,S_{n}(213, 231)} and \eqref{F3,S_{n}(213, 231)}, we obtain  the desired result.
\end{proof}
\begin{cor}
	Let $p=q=1$, then setting three out of the four variables $u$, $v$, $s$ and $t$ equal to one individually in~\eqref{F-213, 231}, we obtain single distributions of  $\lrmax$, $\rlmax$, $\lrmin$ and $\rlmin$  over $S_{n}(213, 231)$:
	\begin{eqnarray}
	\sum_{n\geq 0}\ \sum_{\pi \in S_{n}(213, 231)}x^nu^{\lrmax(\pi)}&=&\frac{1 - 2 x + u x^2}{(1 - 2 x) (1 - u x)};\label{213, 231-lrmax}\\
	\sum_{n\geq 0}\ \sum_{\pi \in S_{n}(213, 231)}x^nv^{\rlmax(\pi)}&=&\frac{1 - x}{1 - x - v x};\label{213, 231-rlrmax}\\
	\sum_{n\geq 0}\ \sum_{\pi \in S_{n}(213, 231)}x^ns^{\lrmin(\pi)}&=&\frac{1 - 2 x + s x^2}{(1 - 2 x) (1 - s x)};\label{213, 231-lrmin}\\
	\sum_{n\geq 0}\ \sum_{\pi \in S_{n}(213, 231)}x^nt^{\rlmin(\pi)}&=&\frac{1 - x}{1 - x - t x}.\label{213, 231-rlmin}
	\end{eqnarray}
\end{cor}

\begin{rem}
The distributions in (\ref{213, 231-lrmax}) and~(\ref{213, 231-lrmin}) (resp., (\ref{213, 231-rlrmax}) and~(\ref{213, 231-rlmin})) are the same because the set $S_{n}(213, 231)$ is invariant under the complement operation, and applying complement exchanges the sets of left-to-right maxima  and left-to-right minima (resp., right-to-left maxima and right-to-left minima).
\end{rem}

\subsection{Permutations in $S_{n}(213, 312)$}
	We first describe the structure of a $(213, 312)$-avoiding permutation.
	Let $\pi=\pi_1\cdots\pi_n  \in S_n(213, 312)$. If $\pi_{i}=n$ then $\pi_{1}< \pi_{2}<\cdots<\pi_{i-1}$   in order to avoid 213. On the other hand, in order to avoid 312, $\pi_{i+1}> \pi_{i+2}>\cdots>\pi_{n}$. We use the structure of $\pi$ to prove the following theorems.
\begin{thm}\label{gf3-(213, 312)}
	For $S_{n}(213, 312)$, we have
	\begin{equation}\label{213, 312-MND-MNA}
	G_{(213, 312)}(x, p, q, y, z)=\frac{A}{p^4 x^4 y^2 + (-1 + q^2 x^2 z)^2 - 2 p^2 x^2 y (1 + q^2 x^2 z)}, 
	\end{equation}
	where $A=(1 - p^3 x^3 y^2 + q x z - q^2 x^2 z - q^3 x^3 z^2 + 
	p^2 x^2 y (-1 + q x z) + p x y (1 + 2 q x z + q^2 x^2 z))$.
\end{thm}

\begin{proof}
	Let $\pi=\pi_1\cdots\pi_n  \in S_n(213, 312)$. If $\pi_{i}=n$ then $\pi_{1}< \pi_{2}<\cdots<\pi_{i-1}$   in order to avoid 213. On the other hand, in order to avoid 312, $\pi_{i+1}> \pi_{i+2}>\cdots>\pi_{n}$. 

Next, we consider the following cases  based on the parity of $i$.
If $i=2k$, $k\ge 1$, we  obtain $\binom{n-1}{2k-1}$ permutations with $k$ non-overlapping ascents and  $\lfloor\frac{n-2k+1}{2}\rfloor$ non-overlapping descents. 
If $i=2k+1,k\ge 0$, we obtain $\binom{n-1}{2k}$ permutations with $k$ non-overlapping ascents and $\lfloor\frac{n-2k}{2}\rfloor$ non-overlapping descents.
So we have
\begin{align}
	&G_{(213, 312)}(x, p, q, y, z)=1+\sum_{n=1}^{\infty}{\sum_{k=1}^{\lfloor n/2 \rfloor}\binom{n-1}{2k-1}x^ny^kz^{\lfloor\frac{n-2k+1}{2}\rfloor}p^{2k-1}q^{n-2k}}+\notag\\&
	\sum_{n=1}^{\infty}{\sum_{k=0}^{\lfloor (n+1)/2 \rfloor}\binom{n-1}{2k}x^ny^kz^{\lfloor\frac{n-2k}{2}\rfloor}p^{2k}q^{n-2k-1}}.\notag
\end{align}
By using MATHEMATICA, we  simplify 	$G_{(213, 312)}(x, p, q, y, z)$ and obtain  (\ref{213, 312-MND-MNA}).
\end{proof}
\begin{cor}
	Setting three out of the four variables $y$, $z$, $p$ and $q$ equal to one respectively in~\eqref{213, 312-MND-MNA}, we obtain single distributions of  $\asc$, $\des$, $\MNA$ and $\MND$ over $S_{n}(213, 312)$:
	\begin{eqnarray}
	\sum_{n\geq 0}\ \sum_{\pi \in S_{n}(213, 312)}x^np^{\asc(\pi)}&=&\frac{1 - p x}{1 - x - p x};\label{213, 312-asc}\\
	\sum_{n\geq 0}\ \sum_{\pi \in S_{n}(213, 312)}x^nq^{\des(\pi)}&=&\frac{1 - q x}{1 - x - q x};\label{213, 312-des}\\	
	\sum_{n\geq 0}\ \sum_{\pi \in S_{n}(213, 312)}x^ny^{\MNA(\pi)}&=&\frac{x - x^2 + x^2 y}{1 - 2 x + x^2 - x^2 y};\label{cor-gf2-S_{n}(213, 312)MNA}\\
	\sum_{n\geq 0}\ \sum_{\pi \in S_{n}(213, 312)}x^nz^{\MND(\pi)}&=&\frac{x - x^2 + x^2 z}{1 - 2 x + x^2 - x^2 z}.\label{cor-gf2-S_{n}(213, 312)MND}	
	\end{eqnarray}
\end{cor}

\begin{rem}
The distributions in (\ref{213, 312-asc}) and~(\ref{213, 312-des}) (resp., (\ref{cor-gf2-S_{n}(213, 312)MNA}) and~(\ref{cor-gf2-S_{n}(213, 312)MND})) are the same because the set $S_{n}(213, 312)$ is invariant under the reverse operation, and applying reverse exchanges ascents and descents (resp., non-overlapping ascents and non-overlapping descents). 
\end{rem}

\begin{thm}\label{thm-F-213, 312}
	For $S_{n}(213, 312)$, we have
	\begin{align}\label{F-213, 312}
	&F_{(213, 312)}(x, p, q, u, v, s, t) =1 + xuvst + \frac{p q s t^2 u^2 v^2 x^3}{(-1 + p t u x) (-1 + p u x + q v x)}+\notag\\& \frac{q s^2 t u v^2 x^2}{1 - q s v x}+\frac{p s t^2 u^2 v x^2}{1 - p t u x} +\frac{	p q s^2 t u^2 v^2 x^3}{(-1 + p u x + q v x) (-1 + q s v x)}.
	\end{align}
\end{thm}

\begin{proof}
	Let  $\pi=\pi_1\cdots\pi_n  \in S_n(213, 312)$.  If $n=0$, we have the term of 1 in $F_{(213, 312)}(x, p, q, u, v, s, t)$. If $\pi\in S_1$,  the g.f. is $xuvst$. For $n\ge 2$, suppose that $\pi_1=i$, $\pi_k=n$ and $\pi_n=j$. We consider the following cases.

	If $i=n$, namely $k=1$, then we have
		$$F_{(213, 312)}(x, p, q, u, v, s, t) =\sum_{n= 2}^{\infty}x^nq^{n - 1}uv^ns^nt.$$
	
	If $j=n$, namely $k=n$, then we have 	$$F_{(213, 312)}(x, p, q, u, v, s, t) =\sum_{n= 2}^{\infty}x^np^{n - 1}u^nvst^n.$$
	
	Next, let $2\le i,j,k\le n-1$.
	If $\pi_1=1$, in order to avoid 312, there are $\binom{n - j - 1}{n - k - 1}$ permutations whose  g.f. is $x^np^{k-1}q^{n - k}u^k
	v^{n - k + 1}st^j$, so the g.f. in this case is
	$$\sum_{n= 2}^{\infty}{\sum_{j= 2}^{n-1}{\sum_{k=j}^{n-1}\binom{n - j - 1}{n - k - 1}x^np^{k-1}q^{n - k}u^k
			v^{n - k + 1}st^j}}.$$
	
	If $\pi_1\ne 1$, in order to avoid 213, there are $\binom{n - i - 1}{k-2}$ permutations whose  g.f. is $x^np^{k-1}q^{n - k}u^k
	v^{n - k + 1}s^it$, so the g.f. in this case is 
	$$\sum_{n= 2}^{\infty}{\sum_{i=2}^{n-1}{\sum_{k=2}^{n+1-i}\binom{n - i - 1}{k-2}x^np^{k-1}q^{n - k}u^k
			v^{n - k + 1}s^it}}.$$
	
	In conclusion,
	\begin{align*}
	&F_{(213, 312)}(x, p, q, u, v, s, t) = 
	1 + xtuvs +\sum_{n= 2}^{\infty}x^nq^{n - 1}uv^ns^nt+\sum_{n= 2}^{\infty}x^np^{n - 1}u^nvst^n+\\&\sum_{n= 2}^{\infty}{\sum_{j= 2}^{n-1}{\sum_{k=j}^{n-1}\binom{n - j - 1}{n - k - 1}x^np^{k-1}q^{n - k}u^k
			v^{n - k + 1}st^j}}+\\&\sum_{n= 2}^{\infty}{\sum_{i=2}^{n-1}{\sum_{k=2}^{n+1-i}\binom{n - i - 1}{k-2}x^np^{k-1}q^{n - k}u^k
			v^{n - k + 1}s^it}}
	\end{align*}

	By using MATHEMATICA, we simplify 	$F_{(213, 312)}(x, p, q, u, v, s, t)$ and obtain (\ref{F-213, 312}).
\end{proof}
\begin{cor}
	Let $p=q=1$, then setting three out of the four variables $u$, $v$, $s$ and $t$ equal to one individually in~\eqref{F-213, 312}, we obtain single distributions of  $\lrmax$, $\rlmax$, $\lrmin$ and $\rlmin$  over $S_{n}(213, 312)$:
	\begin{eqnarray}
	\sum_{n\geq 0}\ \sum_{\pi \in S_{n}(213, 312)}x^nu^{\lrmax(\pi)}&=&\frac{1 - x}{1 - x - u x};\label{213, 312-lrmax}\\
	\sum_{n\geq 0}\ \sum_{\pi \in S_{n}(213, 312)}x^nv^{\rlmax(\pi)}&=&\frac{1 - x}{1 - x - v x};\label{213, 312-rlrmax}\\
	\sum_{n\geq 0}\ \sum_{\pi \in S_{n}(213, 312)}x^ns^{\lrmin(\pi)}&=&\frac{1 - 2 x + s x^2}{(1 - 2 x) (1 - s x)};\label{213, 312-lrmin}\\
	\sum_{n\geq 0}\ \sum_{\pi \in S_{n}(213, 312)}x^nt^{\rlmin(\pi)}&=&\frac{1 - 2 x + t x^2}{(1 - 2 x) (1 - t x)}.\label{213, 312-rlmin}
	\end{eqnarray}
\end{cor}

\begin{rem}
The distributions in (\ref{213, 312-lrmax}) and~(\ref{213, 312-rlrmax}) (resp., (\ref{213, 312-lrmin}) and~(\ref{213, 312-rlmin})) are the same because the set $S_{n}(213, 312)$ is invariant under the reverse operation, and applying reverse exchanges the sets of left-to-right maxima and right-to-left maxima (resp., right-to-left minima and left-to-right minima). 
\end{rem}

\section{Equidistribution results}\label{equidis-sec}
From Theorems~\ref{gf3-(231, 312)}, \ref{gf3-(213, 231)} and \ref{gf3-(213, 312)}, swapping the variables $p$ and $q$, and $y$ and $z$ in the respective formulas, we obtain \emph{ algebraic} proofs of the following equidistribution results. 

\begin{thm}\label{equidist-thm-1}
	The quadruples of statistics $(\asc,\des,\MNA,\MND)$ and $(\des,\asc,\MND,\MNA)$ are equidistributed on $S_n(231,312)$ for all $n\geq 0$.
\end{thm}

\begin{thm}\label{equidist-thm-0}
	The quadruples of statistics $(\asc,\des,\MNA,\MND)$ and $(\des,\asc,\MND,\MNA)$ are equidistributed on $S_n(213,231)$ for all $n\geq 0$.
\end{thm}

\begin{thm}\label{MNA-MND-equid-213-312}
	The quadruples of statistics $(\asc,\des,\MNA,\MND)$ and $(\des,\asc,\MND,\MNA)$ are equidistributed on $S_n(213,312)$ for all $n\geq 0$.
\end{thm}

\begin{thm}\label{equidist-thm-2}
	The quadruple of statistics $(\asc,\des,\MNA,\MND)$ on $S_n(231,312)$ has the same distribution as  $(\des,\asc,\MND,\MNA)$ on $S_n(213,231)$. 
	
\end{thm}

\begin{thm}\label{equidist-thm-3}
	The quadruple of statistics $(\asc,\des,\MNA,\MND)$ is equidistributed on $S_n(231,312)$ and $S_n(213,231)$. 
	\end{thm}

In this section we provide  combinatorial proofs of the five theorems. The combinatorial proofs of Theorems~\ref{equidist-thm-0} and~\ref{MNA-MND-equid-213-312} are trivial: in Theorem~\ref{equidist-thm-0} we can apply the complement operation to  permutations in $S_n(213, 231)$, and in Theorem~\ref{MNA-MND-equid-213-312} we can apply the reverse operation to permutations in $S_n(213, 312)$.

Combinatorial proofs of Theorems~\ref{equidist-thm-1}, \ref{equidist-thm-2} and~\ref{equidist-thm-3} are much more involved and they require introduction of two bijective maps $f$ and $g$ in Sections~\ref{map-f-sec} and~\ref{map-g-sec}, respectively. The map $f$, to be introduced next, is shown by us in Lemma~\ref{involution-lem} to be an involution.

\subsection{Map $f$ and its applications}\label{map-f-sec} 
For  $\pi\in S_n(231,312)$ line the elements in $\{1,2,\ldots n\}$ in a row and insert a vertical line between element $x$ and $x+1$ if $\pi$ can be written as $\pi=\pi'\oplus\pi''$ so that $x\in\pi'$ and $x+1$ corresponds to 1 in $\pi''$. For example, for $\pi=124358769(14)(13)(12)(11)(10)$, we have
	$$1|2|34|5|678|9|(10)(11)(12)(13)(14).$$ Clearly, this way to represent permutations in  $S_n(231,312)$ by the increasing permutation $12\cdots n$ with vertical lines inserted between some of the elements is a bijection. Now the function $f: S_n(231,312)\rightarrow S_n(231,312)$ is defined by representing the given permutation $\pi$ as above, then replacing $x(x+1)$ with $x|(x+1)$ and $x|(x+1)$ by $x(x+1)$ for all $x\in\{1,2,\ldots,n-1\}$, that is, by removing the existing vertical lines and inserting new vertical lines in all other places, and then outputting the corresponding permutation. For the representation of the permutation $\pi$ above, the replacement of lines gives  
	$$123|456|7|89(10)|(11)|(12)|(13)|(14)$$
	and hence $f(124358769(14)(13)(12)(11)(10))=3216547(10)98(11)(12)(13)(14).$

\begin{lem}\label{involution-lem} The map $f$ is an involution, i.e.\ $f^2(\pi)=\pi$ for any $\pi\in S_n(231,312)$ and $n\geq 1$.
\end{lem}	

\begin{proof}
Obvious from the definition of $f$. 
\end{proof}

\begin{rem} Any involution is a bijection (a well-known and easily provable fact), hence $f$ is a bijection.
\end{rem}

\begin{rem}
	Using the alternative description of $f$ introduced in Lemma~\ref{involution-lem} we see that $f$ has no fixed points (the vertical lines cannot be in the same places after application of $f$).
\end{rem}

%\begin{lem} The map $f:S_n(231,312)\rightarrow S_n(231,312)$ is a bijection.
%\end{lem}	
%
%\begin{proof}
%	We only need to show that $f$ is injective. Suppose that $f(\pi)=f(\sigma)$ for $\pi\neq\sigma$. Using Lemma~\ref{involution-lem}, $$\pi=f^2(\pi)=f(f(\pi))=f(f(\sigma))=f^2(\sigma)=\sigma$$ which is a contradiction. So, $f$ is injective. 
%\end{proof}

For $\pi=124358769(14)(13)(12)(11)(10)$, $\asc(\pi)=\des(f(\pi))=6$,  $\des(\pi)=\asc(f(\pi))=7$, $\MNA(\pi)=\MND(f(\pi))=3$, and $\MND(\pi)=\MNA(f(\pi))=4$. The notable relations between $\asc$, $\des$, $\MNA$ and $\MND$ in $\pi$ and $f(\pi)$ are not a coincidence as is shown in the following theorem. Note that the set of statistics in Theorem~\ref{equidist-thm-1} cannot be extended by adding more statistics considered in this paper because $\lrmax(\pi)=7$, $\lrmax(f(\pi))=8$, 	$\lrmin(\pi)=1$, 	$\lrmin(f(\pi))=3$,  $\rlmax(\pi)=5$, $\rlmax(f(\pi))=1$, 	$\rlmin(\pi)=7$ and 	$\rlmin(f(\pi))=8$.

Next, we prove Theorem~\ref{equidist-thm-1}.

\begin{proof}
It is easy to see that the bijection $f$ changes  ascents to descents and vice-versa, this means that it interchanges asc and des, and it also interchanges MNA and MND (a run of descents becomes a run of ascents when we apply $f$).  
\end{proof}

\subsection{Map $g$ and its applications}\label{map-g-sec}
Recall that the structure of a permutation $\sigma\in S_n(213,231)$ is $\sigma=\sigma'\oplus(1\ominus\sigma'')$ where $\sigma'$ and $\sigma''$ are $(213,231)$-avoiding, possibly empty, permutations and $\sigma'$ (if non-empty) is increasing. Hence, $\sigma$ can be decomposed \emph{ uniquely} into a sequence of ascending runs ending at right-to-left maxima. Also, the structure of  a permutation $\pi\in S_n(231,312)$ is $\pi=\pi'\oplus(1\ominus\pi'')$ where $\pi'$ and $\pi''$ are, possibly empty, $(231,312)$-avoiding permutations and $\pi''$ (if non-empty) is decreasing. Hence, $\pi$ can be decomposed \emph{ uniquely} into a sequence of decreasing runs beginning at left-to-right maxima. The map $g:S_n(231,312)\rightarrow S_n(213,231)$ is defined as follows: $g(\pi)$ has a right-to-left maximum in position $n+1-i$ if and only if $\pi$ has a left-to-right maximum in position $i$.
%As the base case, $g(\varepsilon)=\varepsilon$. Now, if $\pi=\pi'\oplus(1\ominus\pi'')$ as above then
%\begin{equation*}%\label{def-g-pi}
%g(\pi)=(\pi'')^r\oplus (1\ominus g(\pi')).
%\end{equation*}
For example,  
\begin{equation}\label{example-g} g(124358769(14)(13)(12)(11)(10))=1234(14)(13)56(12)(11)7(10)98.
\end{equation}

Because of the uniqueness of decomposition of $\pi$ (resp., $g(\pi)$) into decreasing (resp., increasing) runs, clearly, the map $g$ is a bijection. Moreover, it is straightforward to see that $\asc(\pi)=\des(g(\pi))$, $\des(\pi)=\asc(g(\pi))$, $\MNA(\pi)=\MND(g(\pi))$ and $\MND(\pi)=\MNA(g(\pi))$ giving us a proof of Theorem~\ref{equidist-thm-2}.

For our example \eqref{example-g}, $\asc(\pi)=\des(g(\pi))=6$,  $\des(\pi)=\asc(g(\pi))=7$, $\MNA(\pi)=\MND(g(\pi))=3$, and $\MND(\pi)=\MNA(g(\pi))=4$. Note that  the set of statistics in Theorem~\ref{equidist-thm-2} cannot be extended by adding more statistics considered in this paper because in \eqref{example-g}, $\lrmax(\pi)=\lrmax(f(\pi))=7$, 	$\lrmin(\pi)=\lrmin(f(\pi))=1$,  $\rlmax(\pi)=\rlmax(f(\pi))=5$ and	$\rlmin(\pi)=7\neq\rlmin(f(\pi))=8$, and the fact that $g(12)=21$ shows that none of the statistics in $\{\lrmax,\lrmin,\rlmax,\rlmin\}$ can be preserved.

\begin{rem}
	We note that $g$ has a single fixed point for each odd $n$ and no fixed points for any even $n$. Indeed, a fixed point must avoid the patterns $213$, $231$ and $312$, and hence $\pi=12\cdots in(n-1)\cdots(i+1)$ for $i\geq 0$ and $g(\pi)=12\cdots (n-i-1)n(n-1)\cdots(n-i)$. Since $\pi=g(\pi)$ we have that $i=\frac{n-1}{2}$ and the observation follows. 
\end{rem}

Finally, we prove Theorem~\ref{equidist-thm-3}.

\begin{proof}
	The map $g(f(\pi))$ proves the statement by Theorems~\ref{equidist-thm-1} and~\ref{equidist-thm-2}.
\end{proof}

\section{Concluding remarks}\label{open-sec}

In this paper, we found the joint distributions of ($\asc$, $\des$, $\lrmax$, $\lrmin$, $\rlmax$, $\rlmin$) and the joint distributions of ($\asc$, $\des$, $\MNA$, $\MND$) on permutations avoiding any two patterns of length~3. All g.f.'s derived in our paper are rational and we provided combinatorial proofs for five equidistribution results observed from the formulas. It is remarkable that we were able to control so many statistics at the same time while deriving explicit distribution results.

Studying (joint) distributions of statistics in other permutation classes, for example, those considered in \cite{Kitaev2011} is an interesting direction of further research.  

\acknowledgements
\label{sec:ack}
The authors are grateful to the anonymous referees for their many useful suggestions. 

\nocite{*}
\bibliographystyle{abbrvnat}
\bibliography{stat-av-3-bib}

\end{document}